\documentclass[12pt,a4paper]{article}
\usepackage[utf8]{inputenc}
\usepackage{amssymb}
\usepackage{amsmath}
\usepackage{amsthm}
\numberwithin{equation}{section}
\usepackage{hyperref}
\usepackage{pbox}
\usepackage{graphicx}
\usepackage{parskip}

\usepackage{caption}
\usepackage{subcaption}
\usepackage{float}
\usepackage{pax}
\usepackage{pdfpages}

\usepackage{tikz-cd}

\theoremstyle{plain}
\newtheorem{theorem}{Theorem}[section]

\theoremstyle{definition}

\theoremstyle{plain}
\newtheorem{proposition}[theorem]{Proposition}

\theoremstyle{plain}

\DeclareMathOperator{\rank}{rank}

\title{Halfcanonical Gorenstein curves of codimension four}
\author{Patience Ablett}
\date{}

\begin{document}

\maketitle

\begin{abstract}
    Recent work of Schenck, Stillman and Yuan \cite{schenck2020calabiyau} outlines all possible Betti tables for Artin Gorenstein algebras $A$ with regularity($A$) = 4 = codim($A$). We populate the second half of this list with examples of stable curves, and ask if there are further possible constructions. The problem of deformation between curves with the same Hilbert series but different Betti tables is ongoing work, but our work solves one case: a deformation (due to Jan Stevens) between a reducible curve corresponding to Betti table type \hyperref[2.7]{2.7} in \cite{schenck2020calabiyau} and the curve obtained as the intersection of a del Pezzo surface of degree 5 and a cubic hypersurface.
\end{abstract}

\section{Introduction}
Gorenstein rings are a frequently seen subset of Cohen--Macaulay rings, first introduced by Grothendieck in a 1961 seminar.  Since Buchsbaum and Eisenbud's 1977~\cite{10.2307/2373926} structure theorem on Gorenstein rings of codimension 3, much work has been done on the case of Gorenstein rings of codimension 4, including Reid's general structure theorem~\cite{reid2015gorenstein}. 
Recent work on Gorenstein rings involves the study of Gorenstein Calabi--Yau 3-folds, hereon referred to as GoCY 3-folds. Calabi-Yau 3-folds play a vital role in the study of string theory \cite{candelas1985vacuum}, \cite{candelas1991pair}. Following work of Coughlan, Golebiowski, Kapustka and Kapustka~\cite{coughlan2016arithmetically} which presented a list of nonsingular GoCY 3-folds, Schenck, Stillman and Yuan~\cite{schenck2020calabiyau} outlined all possible Betti tables for Artin Gorenstein algebras with Castelnuovo--Mumford regularity and codimension 4. More recently, Kapustka, Kapustka, Ranestad, Schenck, Stillman and Yuan \cite{kapustka2021quaternary} exhibit liftings to GoCY 3-folds corresponding to types of nondegenerate quartic. Other recent work on GoCY 3-folds can be seen in \cite{brown2017polarized}, \cite{brown2019gorenstein}. The possible Betti tables in \cite{schenck2020calabiyau} are split into two sections: eight Betti tables corresponding to the 11 GoCY 3-folds outlined in~\cite{coughlan2016arithmetically}, and eight which cannot correspond to a nonsingular GoCY 3-fold. \\

Our results follow on from~\cite{schenck2020calabiyau}, and focus on those Betti tables which cannot correspond to nonsingular GoCY 3-folds. Our initial goal is to populate the list with concrete examples of stable curves in $\mathbb{P}^5$ corresponding to the given Betti tables. From the restrictions on Castelnuovo-Mumford regularity such curves are halfcanonical, meaning $\omega_C=\mathcal{O}_C(2A)$, where $A$ is the hyperplane class. Our results are summarised in table \ref{tab:table11}. MAGMA code for all eight types can be found at \newline $\hspace*{5mm}$ \url{https://sites.google.com/view/patience-ablett/msc-project}. \newline Note that type 2.4 is described in \cite{schenck2020calabiyau}. We then seek to answer the question of whether these curves are the only possible constructions, and whether we can construct flat deformations between curves in the same Hilbert scheme. Partial results have been achieved here. For type \hyperref[2.7]{2.7} we outline a flat deformation to a curve given by a del Pezzo surface of degree five intersecting a cubic hypersurface. \\

Our method to construct curves relies on first identifying the possible quadric generators in $I_C$. For types \hyperref[2.1]{2.1} and \hyperref[2.2]{2.2} these quadrics are necessarily of the stated form. For types \hyperref[2.5]{2.5} and \hyperref[2.6]{2.6} we can show that we have identified all possible quadrics in the case that they define a Koszul algebra, but a question remains of whether there is a possible set of non-Koszul quadratic generators. Our construction techniques use ideas from liaison theory, which began with work of Peskine and Szpir\'o \cite{peskine1974liaison}. In particular the following result is used: 
\begin{theorem}[\cite{migliore2002liaison} page 77]\label{liaison}
    Let $X_1$, $X_2 \subset \mathbb{P}^n$ be projectively Cohen--Macaulay subschemes of codimension $r$. Then if $X=X_1 \cup X_2$ is Gorenstein it follows that $X_1 \cap X_2$ is also Gorenstein, and of codimension $r+1$.
\end{theorem}
In the situation of the above theorem, we say that $X_1$ and $X_2$ are geo\-metrically G-linked by $X$, and that $X_1$ is residual to $X_2$ in $X$. Suppose more generally that $X_1 \cup X_2 \subset X$. Then if $(I_X:I_{X_1})=I_{X_2}$ and $(I_X:I_{X_2})=I_{X_1}$ we say that $X_1$ and $X_2$ are algebraically G-linked by $X$, and again $X_1$ is residual to $X_2$ in $X$~\cite[pages 62--64]{migliore2002liaison}. \\

Our constructions also rely on the Tom and Jerry formats as seen in \cite{brown2012fano}, \cite{brown2018tutorial} and \cite{papadakis2001gorenstein}. In this paper, for a given ideal $I$ we use $\text{Tom}_i$ to refer to a skew-symmetric matrix with $a_{kl} \in I$ for $k,l \neq i$ and other elements general. Similarly, we define $\text{Tom}_{ij}$ to have $a_{kl} \in I$ for $k,l \notin \{i,j\}$ and other elements general. On the other hand, $\text{Jer}_{ij}$ refers to a skew matrix with $a_{kl} \in I$ for $k$ or $l \in \{i,j\}$ and other elements general. \\

For simplicity we focus mostly on simpler cases of stable curves where a curve $C$ is given by $C_1 \cup C_2$ with $C_1$, $C_2$ nonsingular and irreducible, meeting transversally in $d$ points. Note that the inclusion $i\colon C_1 \rightarrow C$ is a finite morphism. We may therefore use the following proposition:

\begin{proposition}[\cite{MR0463157} Ch. III, Ex. 7.2]
    Let $\pi\colon Y \rightarrow X$ be a finite morphism of projective schemes, with $\dim Y=\dim X$. Then $\omega_Y=\textup{Hom}_{\mathcal{O}_X}(\pi_*\mathcal{O}_Y,\omega_X)$.
\end{proposition}
It follows that in the case of our inclusion, we have $\omega_{C_1}=\text{Hom}_{\mathcal{O}_C}(\mathcal{O}_{C_1},\omega_C)$. Moreover, since $C_1$ is assumed to be nonsingular, it is normal and $\omega_{C_1}=\mathcal{O}_{C_1}(K_{C_1})$. Since $\omega_C=\mathcal{O}_C(2A)$, we are considering $\mathcal{O}_C$-module homomorphisms from $\mathcal{O}_{C_1}$ to $\mathcal{O}_C(2A)$. Such a homomorphism is defined by where $1$ is sent. Indeed, $1$ can mapped to any element of $\mathcal{O}_C(2A)$ annihilated by $I_{C_2}$. Therefore  $\omega_{C_1}=\mathcal{O}_{C_1}(2A_1-D)$, where $A_1$ is the hyperplane class on $C_1$ and $D$ is the locus of double points.  \\

This paper is part of a work in progress with Miles Reid, Jan Stevens and Stephen Coughlan. In particular we hope to publish further work on the question of deformations. Table \ref{tab:table12} outlines which curves lie in the same Hilbert scheme and which we would therefore hope to construct deformations between.

\renewcommand{\arraystretch}{1.3}
\begin{table}
\begin{center}
\begin{tabular}[c]{|l|l|l|}
\hline
 \multicolumn{3}{| c |}{Classifying curves by genus and degree}\\
 \hline
 \rule{0pt}{35pt}\pbox{2.5cm}{Degree\\} & \pbox{2.5cm}{Genus\\} & \pbox{2.5cm}{Corresponding \\ Betti table\\}\\
 \hline
 14 & 15 & CGKK 1 \\
 15 & 16 & CGKK 2, SSY 2.7, SSY 2.8 \\
 16 & 17 & CGKK 3, SSY 2.3, SSY 2.4, SSY 2.6 \\
 17 & 18 & CGKK 4, CGKK 5, CGKK 6, SSY 2.2, SSY 2.5 \\
 18 & 19 & CGKK 7, CGKK 8, SSY 2.1 \\
 19 & 20 & CGKK 9, CGKK 10 \\
 20 & 21 & CGKK 11 \\
 \hline
\end{tabular}
\end{center}
\caption{A summary of which Hilbert scheme every curve lies in.}
\label{tab:table12}
\end{table}
\begin{center}
\end{center}

\section{Examples of stable curves}\label{results}
In this section we present a series of stable curves with free resolutions corresponding to the Betti tables of type 2 in~\cite{schenck2020calabiyau}. We begin by outlining two possible constructions for type \hyperref[2.6]{2.6}, which use techniques from liaison theory and Brown and Reid's Tom and Jerry format. We then outline \hyperref[2.7]{2.7} and describe the flat deformation from type \hyperref[2.7]{2.7} to CGKK 2 \cite{coughlan2016arithmetically}, which lies in the same Hilbert scheme. We finally present type \hyperref[2.3]{2.3}, since this is a somewhat different case which uses rational scrolls. Other constructions are similar to types \hyperref[2.6]{2.6} and \hyperref[2.7]{2.7} and we therefore relegate them to an appendix.

\renewcommand{\arraystretch}{1.3}
\begin{table}
\begin{tabular}[c]{|c||c|c|c|c|}
\hline
 \multicolumn{5}{| c |}{Nodal curve models for each type}\\
 \hline
 \rule{0pt}{35pt}\pbox{2.5cm}{Betti table\\} & \pbox{2.5cm}{Irreducible \\ components\\} & \pbox{2.5cm}{Degrees of \\ components\\} & \pbox{2.5cm}{Genera of\\ components\\} & \pbox{2.5cm}{Number of \\double points\\}\\
 \hline
 Type 2.1 & $C_1 \cup C_2$ & 12, 6 & 10, 4 & 6 \\
 Type 2.2 & $C_1 \cup C_2$ & 11, 6 & 9, 4 & 6\\
 Type 2.3 & $C_1 \cup C_2$ & 9, 7& 7, 5& 6\\
 Type 2.5 & $C_1 \cup C_2$ & 13, 4 & 12, 3 & 4\\
 Type 2.6 & $C_1 \cup C_2$  & 12, 4 or 8, 8 & 11, 3 or 7, 7 & 4  \\
 Type 2.7 & $C_1 \cup C_2$  & 11, 4 & 10, 3 & 4 \\
 Type 2.8 & $C_1 \cup C_2 \cup C_3$ & 7, 4, 4 & 4, 3, 3 & 8\\
 \hline
\end{tabular}
\caption{A summary of our constructions.}
\label{tab:table11}
\end{table}

\subsection{Type 2.6}\label{2.6}
We first construct a curve in $\mathbb{P}^5_{\left<x_0\dots x_5\right>}$ corresponding to Schenck, Stillman and Yuan's type 2.6. 
\begin{table}[h!]
    \[\begin{array}{c|l}
        & 0 \hspace{0.65cm} 1\hspace{0.65cm} 2 \hspace{0.68cm} 3 \hspace{0.65cm} 4 \\ \hline
       0 &1 \hspace{0.5cm} - \hspace{0.45cm} - \hspace{0.45cm} - \hspace{0.45cm} - \\
      1& - \hspace{0.55cm} 4 \hspace{0.65cm} 4 \hspace{0.65cm} 1 \hspace{0.6cm}-\\
      2&- \hspace{0.55cm} 4 \hspace{0.65cm} 8 \hspace{0.65cm} 4 \hspace{0.6cm} -\\
      3&- \hspace{0.55cm} 1 \hspace{0.65cm} 4 \hspace{0.65cm} 4 \hspace{0.6cm} - \\
      4& - \hspace{0.4cm} - \hspace{0.44cm} - \hspace{0.42cm} - \hspace{0.53cm} 1
    \end{array}
\]
\caption*{Type 2.6~\cite{schenck2020calabiyau}}
\label{tab:table5}
\end{table}
The curve $C$ has degree 16, and due to the assumptions on Castelnuovo--Mumford regularity it is halfcanonical with arithmetic genus 17. Let $J=(Q_1,Q_2,Q_3,Q_4)$ be the ideal of the quadric relations, $S=k[x_0,\dots,x_n]$. Then $R=S/J$ has a minimal free resolution with linear part corresponding to the first line of the Betti table. We can use this to rule out possible ideals $J$ where there are too many or too few linear syzygies. Note that we may also have syzygies of higher order, but focusing on the linear syzygies is often enough to find appropriate quadric relations. For type 2.5 and 2.6 we obtain possible quadric relations through an analysis of the case where $R$ is a Koszul algebra, detailed at the end of this section. This raises the question of whether there exist appropriate quadric relations which do not define a Koszul algebra.

It can be shown that the quadrics $\{x_0x_5,x_1x_5,x_2x_5,Q_4\}$, where $Q_4$ is in the ideal $(x_0,x_1,x_2)\backslash(x_5)$, have four linear first syzygies and one linear second syzygy. In this case our curve $C$ breaks into two pieces: $C_1 \subset \mathbb{P}^4_{\left<x_0\dots x_4\right>}$ and $C_2 \subset \mathbb{P}^2_{\left<x_3:x_4:x_5\right>}$, meeting transversally in $d$ points. For simplicity we assume these curves are nonsingular and irreducible.
It follows that $C_2$ is a plane curve defined by an irreducible cubic or quartic. \\ Recall that 
\begin{equation}
   \mathcal{O}_{C_1}(K_{C_1}) = \omega_{C_1}=\mathcal{O}_{C_1}(2A_1-D).
\end{equation}
It follows that $K_{C_1}=2A_1-D$, where $A_1$ is the hyperplane class in $C_1$ and $D$ is the locus of double points of $C_1 \cup C_2$. Hence, deg $K_{C_1} = 2g_1-2=2d_1-d$ and consequently 
\begin{equation}\label{points1}
    d_1=g_1-1+\tfrac{d}{2}.
\end{equation}

Similarly for $C_2$ we have
\begin{equation}\label{points2}
    d_2=g_2-1+\tfrac{d}{2}.
\end{equation}

If $C_2$ was a nonsingular cubic then it would intersect the hypersurface given by $x_5=0$ in a maximum of 3 points. However, from (\ref{points2}) we obtain $d=6$, a contradiction. Consequently $C_2$ is defined by a nonsingular quartic of degree 4 and genus 3. It follows that the double locus of $C$ contains 4 points, and we expect $C_1$ to be a curve of degree 12 and genus 11. \\

Let $\Gamma$ be the curve $C_1 \cup l_1$ in $\mathbb{P}^4$, with $l_1$ defined by $x_0=x_1=x_2=0$. Then from our earlier discussion of nodal curves we have $K_{\Gamma}|_{C_1}=K_{C_1}+D$ where $D$ is the divisor of the double locus of 4 points on $C_1$. Considering $D$ as the divisor of the 4 points on $l_1$, we also have $K_{\Gamma}|_{l_1}=K_{l_1}+D=-2H+D$ where $H$ is the hyperplane class. It follows that $\mathcal{O}_{l_1}(K_{\Gamma})=\mathcal{O}_{l_1}(-2+4)=\mathcal{O}_{l_1}(2)$, so $\Gamma$ is halfcanonical, hence Gorenstein. Thus $\Gamma$ is defined by Pfaffians~\cite{10.2307/2373926} and has degree 13. According to the Betti table we need one more quadric relation and 4 cubic relations so it follows that $\Gamma$ should be defined by the $4 \times 4$ Pfaffians of a $5 \times 5$ skew-symmetric matrix. We now describe some constraints on the matrix to ensure $l_1 \subset \Gamma$, and so that it defines four cubic Pfaffians and one quadric. \\

Consider the matrix 
\begin{equation}
    N =  \begin{pmatrix}
    & a_{12} & a_{13} & a_{14} & a_{15} \\
    & & a_{23} & a_{24} & a_{25}  \\
    & & & a_{34} & a_{35}  \\
    & & & & a_{45}  \\
    \end{pmatrix}.
\end{equation}
Further, let the degrees of the $a_{ij}$ be given by
\begin{equation}
    N =  \begin{pmatrix}
    & 1 & 1 & 1 & 2 \\
    & & 1 & 1 & 2  \\
    & & & 1 & 2 \\
    & & & & 2  \\
    \end{pmatrix}.
\end{equation}
Then the $4 \times 4$ Pfaffians are of degree $(2,3,3,3,3)$. Let $I$ be the ideal of the $4 \times 4$ Pfaffians of $N$. Then
\begin{equation}
\begin{split}
    I = (& a_{12}a_{34}-a_{13}a_{24}+a_{14}a_{23}, \\
    & a_{12}a_{35}-a_{13}a_{25}+a_{15}a_{23}, \\ & a_{12}a_{45}-a_{14}a_{25}+a_{15}a_{24}, \\ & a_{13}a_{45}-a_{14}a_{35}+a_{15}a_{34}, \\ & a_{23}a_{45}-a_{24}a_{35}+a_{25}a_{34}).
    \end{split}
\end{equation}
It follows that for any $\{k,l\} \subset \{1,\dots,5\}$, $k\neq l$, setting $a_{ij} \in J=(x_0,x_1,x_2)$ if $i \in \{k,l\}$ or $j \in \{k,l\}$ ensures $I \subset J$. The remaining elements may be general in the coordinates of $\mathbb{P}^4$. In other words, $N$ is a variant of $\text{Jer}_{kl}$. Similarly, if we set $a_{ij} \in J$ for $i,j \neq k$, we ensure $I \subset J$, in which case $N$ is a variant of $\text{Tom}_k$.\\

Defining $\Gamma$ in this way ensures it breaks into two irreducible nonsingular components, our line $l_1$ and curve $C_1$. Moreover, $l_1$ and $C_1$ intersect in 4 points which define a quartic, $q_4$. Mapping $q_4$ into $\mathbb{P}^2$ by adding arbitrary terms in $(x_5)$ defines a nonsingular quartic curve $C_2$. $C=C_1 \cup C_2 \subset \mathbb{P}^5$ is a codimension 4 Gorenstein curve corresponding to Betti table 2.6. A computer algebra package such as MAGMA can be used to verify that each curve is nonsingular and that $C_1$ and $C_2$ intersect transversally. We can also use MAGMA to compute the free resolution as a sanity check. \\

Now instead suppose that the four quadrics are given by $(x_0,x_1) \cap (x_2,x_3)$, which again have the correct minimal free resolution. It follows that $C$ breaks up into $C_1 \subset \mathbb{P}^3_{\left<x_2\dots x_5\right>}$ and $C_2 \subset \mathbb{P}^3_{\left<x_0:x_1:x_4:x_5\right>}$. We may define $C_1$ and $C_2$ in the following way. Consider the complete intersection $X_1=V(F_1,F_2) \subset \mathbb{P}^3_{\left<x_2\dots x_5\right>}$ given by cubics $F_1$, $F_2$ containing the line $l_1 \colon x_2=x_3=0$ in $\mathbb{P}^3$. Such cubics have the form 
\begin{equation}
    F_1=x_2P_3+x_3P_4, \quad F_2=x_2Q_3 + x_3Q_4,
\end{equation}
with $P_3,Q_3,P_4,Q_4$ quadratic forms in $k[x_2,\dots,x_5]$. 
Then $X_1$ breaks into two irreducible components, namely the line $l_1$ and the curve $C_1$, defined by $(F_1,F_2,P_3Q_4-P_4Q_3)$. The curve $C_1$ is nonsingular with degree 8 and genus 7. Similarly we are able to define another (3,3) complete intersection $X_2=V(F_3,F_4) \subset \mathbb{P}^3_{\left<x_0:x_1:x_4:x_5\right>}$ containing the line $l_2 \colon x_0=x_1=0$ in $\mathbb{P}^3$:
\begin{equation}
    F_3=x_0P_1+x_1P_2, \quad F_4=x_0Q_1+x_1Q_2.
\end{equation}
Here $P_1,P_2,Q_1,Q_2$ are quadratic forms in $k[x_0,x_1,x_4,x_5]$. Again $X_2$ breaks into two irreducible components with $C_2$ defined by $(F_3,F_4,P_1Q_2-P_2Q_1)$, and $C_2$ is nonsingular with degree 8 and genus 7. It follows from (\ref{points1}), (\ref{points2}) that $C_1$ and $C_2$ meet in four points, which lie on the line $\mathbb{P}^1_{\left<x_4:x_5\right>}$. We outline constraints on the $P_i$, $Q_j$ so that this occurs. If 
\begin{equation}
\begin{split}
P_3|_{x_2=x_3=0}&=P_1|_{x_0=x_1=0}, \\  P_4|_{x_2=x_3=0}&=P_2|_{x_0=x_1=0}, \\  Q_3|_{x_2=x_3=0}&=Q_1|_{x_0=x_1=0}, \\  Q_4|_{x_2=x_3=0}&=Q_2|_{x_0=x_1=0}, 
\end{split}
\end{equation}
then 
\begin{equation}
    R(x_4,x_5)=(P_3Q_4-P_4Q_3)|_{x_2=x_3=0}=(P_1Q_2-P_2Q_1)|_{x_0=x_1=0}.
\end{equation} In this situation, $C_1$ and $C_2$ meet in exactly 4 points defined by the quartic $R$ in $\mathbb{P}^1_{\left<x_4:x_5\right>}$. Their union is a Gorenstein codimension 4 curve with Betti table 2.6. \\

The candidates for the quadric generators arise from work of Mantero--Mastroeni \cite{mantero2021betti}. Assuming $R=S/J$ is Koszul, we analyse $J=(Q_1,Q_2,Q_3,Q_4)$ in the context of different heights. If $\text{ht} J=4$ then $J$ is a complete intersection of four quadrics and does not have linear syzygies. If $\text{ht} J=1$ then it is given as $zI$ where $z$ is a linear form and $I$ is a complete intersection of linear forms~\cite{mantero2021betti}. Thus for such $J$, $R$ would not correspond to the type 2.6 Betti table, since there would be too many linear syzygies. Moreover, for $\text{ht} J=3$, $R$ is a Koszul almost complete intersection and thus has at most two linear syzygies~\cite{mastroeni2018koszul}. Hence, $J$ must have height 2. Mantero--Mastroeni show that for a Koszul algebra of four quadrics with $\text{ht}J=2$ to have the required Betti table, it must have multiplicity $e(R)=2$.
\begin{theorem}[Mantero--Mastroeni~\cite{mantero2021betti}]
    Let $R$ be Koszul with $\textup{ht}J=2=e(R)$. Then $J$ has one of the following possible forms: \\
    
    \textup{(I)}  $(x_0,x_1) \cap (x_2,x_3)$ or $(x_0^2,x_0x_1,x_1^2,x_0x_2+x_1x_3)$\\
    
    \textup{(II)} $(a_1x_0,a_2x_0,a_3x_0,q)$ where the $a_i$ are independent linear forms and $q \in (a_1,a_2,a_3)\backslash(x_0)$ \\
    
    \textup{(III)} $(a_1x_0,a_2x_0,a_3x_0,q)$ where the $a_i$ are independent linear forms and $q$ is a non-zero divisor modulo $(a_1x_0,a_2x_0,a_3x_0)$. \\
    
\end{theorem} 
Case (III) does not correspond to a Betti table with four linear first syzygies and one linear second syzygy, so we are in case (I) or case (II). In case (I) the latter option is not reduced so we are restricted to the case $J=(x_0,x_2) \cap (x_2,x_3)$.

\subsection{Type 2.7}\label{2.7}
The following construction is an example of a codimension 4 Gorenstein curve with Betti table as in type 2.7. Any such curve has degree 15 and arithmetic genus 16. \\
\begin{table}[h!]
    \[\begin{array}{c|l}
        & 0 \hspace{0.65cm} 1\hspace{0.65cm} 2 \hspace{0.68cm} 3 \hspace{0.65cm} 4 \\ \hline
       0 &1 \hspace{0.5cm} - \hspace{0.45cm} - \hspace{0.45cm} - \hspace{0.45cm} - \\
      1& - \hspace{0.55cm} 5 \hspace{0.65cm} 5 \hspace{0.65cm} 1 \hspace{0.6cm}-\\
      2&- \hspace{0.55cm} 1 \hspace{0.65cm} 2 \hspace{0.65cm} 1 \hspace{0.6cm} -\\
      3&- \hspace{0.55cm} 1 \hspace{0.65cm} 5 \hspace{0.65cm} 5 \hspace{0.6cm} - \\
      4& - \hspace{0.4cm} - \hspace{0.44cm} - \hspace{0.42cm} - \hspace{0.53cm} 1
    \end{array}
\]
\caption*{Type 2.7~\cite{schenck2020calabiyau}}
\label{tab:table6}
\end{table}

Consider the quadrics 
\begin{equation}
    Q_1=x_0x_5, \quad Q_2=x_1x_5, \quad Q_3=x_2x_5, \quad Q_4, \quad Q_5,
\end{equation}
with $Q_4,Q_5 \in (x_0,x_1,x_2)$. Then $J=(Q_1,Q_2,Q_3,Q_4,Q_5)$ has five linear syzygies as required. Thus, $C$ breaks up into two curves, namely $C_1 \subset \mathbb{P}^4_{\left<x_0\dots x_4\right>}$ and $C_2 \subset \mathbb{P}^3_{\left<x_3:x_4:x_5\right>}$. Again, $C_2$ must be defined by a nonsingular quartic with degree 4 and genus 3, and the double locus of $C$ is 4 points. We obtain from (\ref{points1}) that $C_1$ is degree 11 and genus 10. Let $l_1$ be the line $x_0=x_1=x_2=0$ in $\mathbb{P}^4$. \\

Once more $\Gamma=C_1 \cup l_1$ is Gorenstein since $K_{\Gamma}|_{l_1}$ is halfcanonical. Since $\Gamma$ is degree 12 and we need one more cubic relation we define $\Gamma$ as the complete intersection of two quadrics, $Q_4$ and $Q_5$, and a cubic, $F$, all in $(x_0,x_1,x_2)$. It follows that $C_1$ and $l_1$ are the two irreducible components of $\Gamma$, and $C_1$ is nonsingular. The curves $C_1$ and $l_1$ meet in 4 points defining a quartic and mapping this quartic into $\mathbb{P}^2$, adding arbitrary terms in $(x_5)$, defines a nonsingular quartic curve $C_2$. The union of these two curves is Gorenstein codimension 4, with Betti table as prescribed. \\

Moreover, we can construct a deformation to CGKK 2~\cite{coughlan2016arithmetically}, which is in the same Hilbert scheme as type 2.7 and type 2.8, courtesy of Jan Stevens.
\begin{table}[h!]
    \[\begin{array}{c|l}
        & 0 \hspace{0.65cm} 1\hspace{0.65cm} 2 \hspace{0.68cm} 3 \hspace{0.65cm} 4 \\ \hline
       0 &1 \hspace{0.5cm} - \hspace{0.45cm} - \hspace{0.45cm} - \hspace{0.45cm} - \\
      1& - \hspace{0.55cm} 5 \hspace{0.65cm} 5 \hspace{0.5cm} - \hspace{0.45cm}-\\
      2&- \hspace{0.55cm} 1 \hspace{0.5cm} - \hspace{0.5cm} 1 \hspace{0.6cm} -\\
      3&- \hspace{0.4cm} - \hspace{0.5cm} 5 \hspace{0.65cm} 5 \hspace{0.6cm} - \\
      4& - \hspace{0.4cm} - \hspace{0.44cm} - \hspace{0.42cm} - \hspace{0.53cm} 1
    \end{array}
\]
\caption*{CGKK 2~\cite{coughlan2016arithmetically}}
\label{tab:table10}
\end{table}

Consider the syzygy module of the five generating quadrics, 
\begin{equation}
    \begin{split}
    & Q_1=x_0x_5, \\
    & Q_2=x_1x_5, \\
    & Q_3=x_2x_5, \\
    & Q_4=a_1x_0+a_2x_1+a_3x_2, \\ 
    & Q_5=b_1x_0+b_2x_1+b_3x_2.
    \end{split}
\end{equation}
Then this is a $6 \times 5$ matrix

\begin{equation}
    M=\begin{pmatrix}
    0 & x_2 & -x_1 & 0 & 0  \\
    -x_2 & 0 & x_0 & 0 & 0  \\
    x_1 & -x_0 & 0 & 0 & 0  \\
    -b_1 & -b_2 & -b_3 & 0 & x_5  \\
    a_1 & a_2 & a_3 & -x_5 & 0 \\
    0 & 0 & 0 & -Q_5 & Q_4 \\
    \end{pmatrix}.
\end{equation}
Using deformation variable $t$, we construct a new matrix
\begin{equation}
    M_t=\begin{pmatrix}
    0 & x_2 & -x_1 & tb_1 & -ta_1  \\
    -x_2 & 0 & x_0 & tb_2 & -ta_2  \\
    x_1 & -x_0 & 0 & tb_3 & -ta_3  \\
    -b_1 & -b_2 & -b_3 & 0 & x_5  \\
    a_1 & a_2 & a_3 & -x_5 & 0 \\
    0 & 0 & 0 & -Q_5 & Q_4 \\
    \end{pmatrix}.
\end{equation}
Ignoring the bottom row of the matrix, and multiplying rows 4 and 5 by $t$ gives a skew-symmetric matrix. We may then take the $4 \times 4$ Pfaffians and cancel $t$ to obtain the five quadrics
\begin{equation}
    \begin{split}
        & Q_1 = x_0x_5 - t(a_2b_3-a_3b_2), \\
        & Q_2 = x_1x_5 - t(a_1b_3-a_3b_1), \\
        & Q_3 = x_2x_5 - t(a_1b_2-a_2b_1), \\
        & Q_4=a_1x_0+a_2x_1+a_3x_2, \\ 
        & Q_5=b_1x_0+b_2x_1+b_3x_2.
    \end{split}
\end{equation}
We also deform the cubic $F$. Suppose $F=c_1x_0+c_2x_1+c_3x_2$, with $c_1,c_2,c_3$ all of degree 2. As discussed earlier, the quartic in type $2.7$ is given by a quartic in $k[x_0,\dots,x_4]$ plus additional terms in $(x_5)$. In fact the first part is the quartic obtained as the determinant of the matrix \begin{equation}
    N=\begin{pmatrix}
    a_1 & a_2 & a_3 \\
    b_1 & b_2 & b_3 \\
    c_1 & c_2 & c_3 \\
    \end{pmatrix}.
\end{equation}
We can write the quartic as $q=\text{det}(N) + gx_5$ where $g$ is a degree three polynomial. We define our deformed cubic as $F_t=F+tg$. Consider the ideal $I_t=(Q_1, Q_2, Q_3, Q_4,Q_5,F_t,q)$. Then for $t=0$ this is clearly the defining ideal for our type 2.7 nodal curve. Otherwise, note that $tq$ is in the ideal $J_t$ generated by the first six relations. This follows since 
\begin{equation}
tq + c_1Q_1+c_2Q_2+c_3Q_3 -x_5F_t = 0.
\end{equation}
Note that $J_t$ is prime, which can be checked with computer algebra. Thus if $t$ is invertible then $I_t=J_t$ and is defined by the $4 \times 4$ Pfaffians of a $5 \times 5$ skew-symmetric matrix intersecting a cubic hypersurface. This deformation corresponds to Betti table CGKK 2~\cite{coughlan2016arithmetically}.

\subsection{Type 2.3}\label{2.3}
We now focus on a curve with degree 16 and genus 17 corresponding to Betti table type 2.3. \\
\begin{table}[h!]
    \[\begin{array}{c|l}
        & 0 \hspace{0.65cm} 1\hspace{0.65cm} 2 \hspace{0.68cm} 3 \hspace{0.65cm} 4 \\ \hline
       0 &1 \hspace{0.5cm} - \hspace{0.45cm} - \hspace{0.45cm} - \hspace{0.45cm} - \\
      1& - \hspace{0.55cm} 4 \hspace{0.65cm} 3 \hspace{0.52cm} - \hspace{0.45cm}-\\
      2&- \hspace{0.55cm} 3 \hspace{0.65cm} 6 \hspace{0.65cm} 3 \hspace{0.6cm} -\\
      3&- \hspace{0.4cm} - \hspace{0.5cm} 3 \hspace{0.65cm} 4 \hspace{0.6cm} - \\
      4& - \hspace{0.4cm} - \hspace{0.44cm} - \hspace{0.42cm} - \hspace{0.53cm} 1
    \end{array}
\]
\caption*{Type 2.3~\cite{schenck2020calabiyau}}
\label{tab:table8}
\end{table}

Consider the cubic scroll $\mathbb{F}=\mathbb{F}(1,2) \subset \mathbb{P}^4_{\left<x_0\dots x_4\right>}$, defined by equations 
\begin{equation}
    \rank\begin{pmatrix}
    x_0 & x_1 & x_3 \\
    x_1 & x_2 & x_4 \\
    \end{pmatrix} \leq 1.
\end{equation}

$\mathbb{F}(1,2)$ is given by the surface scroll $\mathbb{F}_1$ embedded into $\mathbb{P}^4$ by the linear system $2A+B$ where $A$ is the fibre of $\mathbb{F}_1 \rightarrow \mathbb{P}^1$ and $B$ is the negative section~\cite{Reid1996ChaptersSurfaces}.
The curve $C_2$ is given by $\mathbb{F} \cap X$ where $X$ is a general cubic hypersurface. This curve has degree 9 and genus 7. $C_1 \subset \mathbb{P}^3$ is residual to a conic in a (3,3) complete intersection, such that the double locus of $C$ is given by 6 points. It has degree 7 and genus 5. \\

To describe the construction in terms of explicit relations, consider $x_0x_2-x_1^2$, the third minor of the matrix defining $\mathbb{F}$. $C_1$ lies in $\mathbb{P}^3=V(x_3,x_4)$. Define the plane quadric $Q=V(x_0x_2-x_1^2,x_5) \subset \mathbb{P}^3$, and consider two general cubics $G_1, G_2$ containing $Q$. Such cubics have the form 
\begin{align*}
    G_1 = P_1(x_0x_2-x_1^2)+Q_1x_5, \\
    G_2 = P_2(x_0x_2-x_1^2)+Q_2x_5,
\end{align*}
with $P_1$, $P_2$ linear and $Q_1$, $Q_2$ quadrics. $C_1$ is residual to $Q$ in the $(3,3)$ complete intersection $(G_1,G_2)$. It is defined by one further cubic, given by $H=P_1Q_2-P_2Q_1$. Mapping this $H$ into $\mathbb{P}^4$, adding arbitrary terms in $(x_3,x_4)$, defines a cubic hypersurface $X$. We have $C_2 = \mathbb{F} \cap X$, and $C_1 \cup C_2$ is a Gorenstein codimension four curve in $\mathbb{P}^5$ corresponding to Betti table type 2.3.

\section{Further research}
Having populated the list of Betti tables with examples of nodal curves, a number of open questions remain. Firstly, can we construct more deformations between curves in the same Hilbert scheme, as in type 2.7? Email correspondence with Jan Stevens and Stephen Coughlan answers in the affirmative for types 2.6 and 2.8. There is also the question of whether this list is exhaustive, or if there are further possible curve constructions. The existence of topologically different constructions for type 2.6 suggests this could be the case for other types. In particular for higher degree curves, the picture may be more complicated. In some cases we have only been able to definitively state the quadrics if they define a Koszul algebra, so there may be a different set of quadric relations. We have restricted our search to nodal curves, and so it is possible there are further constructions with worse singularities. We also raise the idea of constructing surfaces and singular 3-folds corresponding to the type 2 Betti tables, or alternatively finite point sets.

\section{Appendix}
\subsection{Type 2.1}
\label{2.1}
We now consider how to construct a curve in $\mathbb{P}^5$ corresponding to Betti table type 2.1, which has degree 18 and genus 19.
\begin{table}[h!]
    \[\begin{array}{c|l}
        & 0 \hspace{0.65cm} 1\hspace{0.65cm} 2 \hspace{0.68cm} 3 \hspace{0.65cm} 4 \\ \hline
       0 &1 \hspace{0.45cm} - \hspace{0.49cm} - \hspace{0.49cm} - \hspace{0.4cm} - \\
      1& - \hspace{0.5cm} 2 \hspace{0.7cm} 1 \hspace{0.54cm} - \hspace{0.4cm}-\\
      2&- \hspace{0.5cm} 9 \hspace{0.6cm} 18 \hspace{0.57cm} 9 \hspace{0.57cm} -\\
      3&- \hspace{0.35cm} - \hspace{0.54cm} 1 \hspace{0.65cm} 2 \hspace{0.57cm} - \\
      4& - \hspace{0.35cm} - \hspace{0.45cm} - \hspace{0.45cm} - \hspace{0.5cm} 1
    \end{array}
\]
\caption*{Type 2.1 \cite{schenck2020calabiyau}}
\label{tab:table2}
\end{table}

We see there are two quadric relations, $Q_1$ and $Q_2$, with one linear syzygy, which may be given as $L_1Q_1-L_2Q_2=0$ for some linear forms $L_1,L_2$. Since we want $Q_1 \neq \lambda Q_2$ for any scalar $\lambda$, we have $L_1 \neq \lambda L_2$ and consequently $L_1$ divides $Q_2$, $L_2$ divides $Q_1$. It follows that $Q_1=L_1L_3$, $Q_2=L_2L_3$ for some linear form $L_3$. Thus without loss of generality we can set
\begin{equation}
Q_1 = x_0x_5, \quad Q_2=x_1x_5.
\end{equation}
It follows that $C$ is singular, and breaks up into $C_1 \subset \mathbb{P}^4_{\left<x_0\dots x_4\right>}$ of degree $d_1$ and genus $g_1$ and $C_2 \subset \mathbb{P}^3_{\left<x_2\dots x_4\right>}$ of degree $d_2$ and genus $g_2$. Supposing that $C$ is at worst a nodal curve, with $C_1$ and $C_2$ nonsingular, and supposing further that $C_2$ is an $(a,b)$ complete intersection we have that $d_2=ab$ and $g_2=\tfrac{1}{2}ab(a+b-4)+1$. It follows from (\ref{points2}) that 
\begin{equation}
    ab=\tfrac{1}{2}ab(a+b-4)+\tfrac{d}{2}.
\end{equation}
Since we have no relations of degree 4 or higher we expect $a,b$ to be at most 3. Notice also that since $C_2$ intersects $x_5=0$ in $ab$ points, $C_1$ and $C_2$ must intersect in at most $ab$ points. This excludes the cases $(1,2),(1,3),(2,2)$, and if $a=b=3$ then we obtain that $d=0$, contradicting our assumption that $C$ is a nodal curve consisting of two nonsingular curves intersecting transversally at a non-zero number of points. Thus we look at the case where $C_2$ is a $(2,3)$ complete intersection. \\ 

It follows that $g_2=4$ and $d_2=6$. We expect $C_1$ to be a curve with $d_1=18-6=12$ and consequently $g_1=10$. Let $Q_3 \in (x_2,x_3,x_4,x_5)$ be the quadric in the complete intersection. Consider the plane conic $q$ in $\mathbb{P}^4$ defined by $x_0=x_1=Q_3|_{x_5=0}=0$. Then $\omega_{q}=\mathcal{O}_{q}(-1)$ by the adjunction formula \cite[page 41]{eisenbud_harris_2016}, and $K_q=-A_q$ where $A_q$ is the hyperplane class. Let $\Gamma$ be the union of $C_1$ and $q$. Then as before $K_{\Gamma}|_q=K_{q}+D$ where $D$ is the divisor of the 6 double points on $q$. Thus $\mathcal{O}_q(K_{\Gamma})=\mathcal{O}_q(-1+3)=\mathcal{O}_q(2)$, and $\Gamma$ is halfcanonical and hence Gorenstein. \\

Since $\Gamma$ is degree 14 and Gorenstein codimension 3 we expect it to be defined by the $6 \times 6$ Pfaffians of a $7 \times 7$ skew-symmetric matrix. All elements of the matrix must be linear, since there are no quartic or higher degree relations. Thus we now consider what conditions must be satisfied for the Pfaffians of the matrix to lie in the ideal $J=(x_0,x_1,Q_3|_{x_5=0})$, but not in $(x_0,x_1)$. Consider the matrix 
\begin{equation}
    M = \begin{pmatrix}
    & a_{12} & a_{13} & a_{14} & a_{15} & a_{16} & a_{17}\\
    & & a_{23} & a_{24} & a_{25} & a_{26} & a_{27} \\
    & & & a_{34} & a_{35} & a_{36} & a_{37} \\
    & & & & a_{45} & a_{46} & a_{47} \\
    & & & & & a_{56} & a_{57} \\
    & & & & & & a_{67} \\
    \end{pmatrix}.
\end{equation}
If $Q_3$ is made up of a reducible part in $(x_2^2,x_2x_3,x_2x_4,x_3^2,x_3x_4,x_4^2)$ plus terms in $(x_5)$ we can construct an $M$ such that the Pfaffians lie in $J$. An open problem following on from this work is whether such an $M$ can be constructed to contain a more general quadric, with $C_1$ still a nonsingular and irreducible curve. If $Q_3$ is in the required form, assume without loss of generality that $Q_3|_{x_5=0}=x_2x_3$. Then the following constraints ensure the Pfaffians lie in $J$. First suppose that $a_{ij} \in (x_0,x_1)$ for $i,j,k \notin \{5,6,7\}$, i.e. that $M$ is a $\text{Tom}_{567}$. We add the further constraints that, except for $a_{56}$ which is general, $a_{ij} \in (x_0,x_1,x_2)$ for $j=5$ and $a_{ij} \in (x_0,x_1,x_3)$ for $j=6$. \\

The curve $C_1$ residual to $q$ in $\Gamma$ is nonsingular and irreducible, of degree 12 and genus 10. Moreover, the intersection of $C_1$ with the conic $q$ is 6 points given as the intersection of $Q_3|_{x_5=0}$ and a cubic $H$ in $\mathbb{P}^2_{\left<x_2\dots x_4\right>}$. We can map this cubic into $\mathbb{P}^3$, adding arbitrary terms in $(x_5)$ to define a new nonsingular cubic. The $(2,3)$ complete intersection $C_2$ is given by $(Q_3,H)$. The union of these two curves defines our nodal curve with resolution given by Betti table 2.1.

\subsection{Type 2.2}\label{2.2}
We now outline a construction of a nodal curve $C \subset \mathbb{P}^5$ with degree 17 and genus 18 corresponding to Betti table type 2.2. \\
\begin{table}[h!]
    \[\begin{array}{c|l}
        & 0 \hspace{0.65cm} 1\hspace{0.65cm} 2 \hspace{0.68cm} 3 \hspace{0.65cm} 4 \\ \hline
       0 &1 \hspace{0.5cm} - \hspace{0.45cm} - \hspace{0.45cm} - \hspace{0.45cm} - \\
      1& - \hspace{0.55cm} 3 \hspace{0.65cm} 1 \hspace{0.52cm} - \hspace{0.45cm}-\\
      2&- \hspace{0.55cm} 5 \hspace{0.6cm} 12 \hspace{0.5cm} 5 \hspace{0.6cm} -\\
      3&- \hspace{0.4cm} - \hspace{0.5cm} 1 \hspace{0.65cm} 3 \hspace{0.6cm} - \\
      4& - \hspace{0.4cm} - \hspace{0.44cm} - \hspace{0.42cm} - \hspace{0.53cm} 1
    \end{array}
\]
\caption*{Type 2.2~\cite{schenck2020calabiyau}}
\label{tab:table7}
\end{table}

The quadric relations in $\mathbb{P}^5_{\left<x_0\dots x_5\right>}$ are necessarily of the form
\begin{equation}
    Q_1 = x_0x_5, \quad Q_2= x_1x_5, \quad Q_3,
\end{equation}
with the single syzygy $x_1Q_1\equiv x_0Q_2$. Thus $C$ breaks up into $C_1 \subset \mathbb{P}^4_{\left<x_0\dots x_4\right>}$ and $C_2 \subset \mathbb{P}^3_{\left<x_2\dots x_5\right>}$. The curve $C_2$ is a $(2,3)$ complete intersection of degree 6 and genus 4. Moreover, the double locus consists of 6 points and we expect $C_1$ to be of genus 11 and degree 9. For $C_1 \cup C_2$ to be halfcanonical we require the 6 points to define a hyperplane section, and we again consider a plane conic $q$ in $\mathbb{P}^4_{\left<x_0\dots x_4\right>}$ in the plane of the double points, with $q$ defined by $x_0=x_1=Q_3=0$. Then we once more have that $\Gamma = C_1 \cup q$ is halfcanonical and hence Gorenstein. We thus define $\Gamma$ using the Pfaffians of a skew-symmetric matrix. The construction is as follows: we define a matrix 
\begin{equation}
    M =  \begin{pmatrix}
    & b_{12} & b_{13} & b_{14} & b_{15} \\
    & & a_{23} & a_{24} & a_{25}  \\
    & & & a_{34} & a_{35}  \\
    & & & & a_{45}  \\
    \end{pmatrix}
\end{equation}
with entries of the following degrees
\begin{equation}
    M =  \begin{pmatrix}
    & 2 & 2 & 2 & 2 \\
    & & 1 & 1 & 1  \\
    & & & 1 & 1  \\
    & & & & 1  \\
    \end{pmatrix},
\end{equation}
whose $4 \times 4$ Pfaffians are four cubics and a quadric, as in §\ref{2.6}. Unlike §\ref{2.6}, we wish for the four cubics to vanish on the plane $x_0=x_1=0$ in $\mathbb{P}^4$, and for the quadric not to lie in $(x_0,x_1)$ - it may be general in $\mathbb{P}^4$. This occurs if we set, for example, the $b_{ij}$ as quadrics in $(x_0,x_1)$, and the $a_{ij}$ linear and in all coordinates $x_0$ to $x_4$. Let $\Gamma$ be the curve defined by the $4 \times 4$ Pfaffians, and set $Q_3$ to be the quadric Pfaffian. Consider the residual curve $C_1$ to the conic defined by $x_0=x_1=Q_3=0$ in $\Gamma$. This is nonsingular and irreducible, and has degree 11 and genus 9 as required. It meets the conic in 6 points, which define a cubic $F_1$ such that the zero locus is given by $x_0=x_1=Q_3=F_1=0$. Mapping $Q_3$ and $F_1$ into $\mathbb{P}^3$ by adding arbitrary terms in $(x_5)$ we obtain the complete intersection $C_2$.

\subsection{Type 2.5}\label{2.5}
We now construct a curve in $\mathbb{P}^5_{\left<x_0\dots x_5\right>}$ corresponding to Schenck, Stillman and Yuan's type 2.5. The curve $C$ has degree 17, and due to the assumptions on Castelnuovo--Mumford regularity it is halfcanonical with arithmetic genus 18. \\
\renewcommand{\arraystretch}{1}
\begin{table}[h!]
    \[\begin{array}{c|l}
        & 0 \hspace{0.65cm} 1\hspace{0.65cm} 2 \hspace{0.68cm} 3 \hspace{0.65cm} 4 \\ \hline
       0 &1 \hspace{0.5cm} - \hspace{0.45cm} - \hspace{0.45cm} - \hspace{0.45cm} - \\
      1& - \hspace{0.55cm} 3 \hspace{0.65cm} 3 \hspace{0.65cm} 1 \hspace{0.6cm}-\\
      2&- \hspace{0.55cm} 7 \hspace{0.55cm} 14 \hspace{0.55cm} 7 \hspace{0.6cm} -\\
      3&- \hspace{0.55cm} 1 \hspace{0.65cm} 3 \hspace{0.65cm} 3 \hspace{0.6cm} - \\
      4& - \hspace{0.4cm} - \hspace{0.44cm} - \hspace{0.42cm} - \hspace{0.53cm} 1
    \end{array}
\]
\caption*{Type 2.5~\cite{schenck2020calabiyau}}
\label{tab:table4}
\end{table}

Let $J=(Q_1,Q_2,Q_3)$ be the ideal of the three quadrics, $S=k[x_0,\dots,x_n]$. In the case that $R=S/J$ is a Koszul algebra we may apply a theorem of Mantero--Mastroeni~\cite{mantero2021betti} to categorize the quadrics.
\begin{proposition}
If $R=S/J$ is a Koszul algebra then $J$ is given by \[Q_1 = x_0x_5, \quad Q_2 = x_1x_5, \quad Q_3 = x_2x_5.\]
\end{proposition}

\begin{proof}
    If $\text{ht}J=2$ then $R$ is an almost complete intersection and consequently has at most two linear syzygies~\cite{mastroeni2018koszul}. If $\text{ht}J=3$ then $R$ is a complete intersection and there are no linear syzygies. If $\text{ht} J=1$ then it is given as $zI$ where $z$ is a linear form and $I$ is a complete intersection of linear forms~\cite{mantero2021betti}. This option has the correct number of syzygies. Thus, without loss of generality, let $z=x_5$, $I=(x_0,x_1,x_2)$. Then $J=zI=(x_0x_5,x_1x_5,x_2x_5).$
\end{proof}
It follows that any curve $C$ with such quadric relations breaks up into two curves, $C_1 \subset \mathbb{P}^4_{\left<x_0\dots x_4\right>}$ and $C_2 \subset \mathbb{P}^2_{\left<x_3:x_4:x_5\right>}$. Thus $C_2$ is defined by a single quadric, cubic or quartic in $\mathbb{P}^2$. \\

Recall that a nonsingular plane quadric has genus 0, and a cubic has genus 1. Assuming $C_2$ is nonsingular, it follows that $C_2$ cannot be defined by a quadric or cubic. Note that $C_1 \cap C_2 \subset C_1 \cap H$ where $H$ is the hyperplane given by $x_5=0$. $C_2$ intersects $H$, and consequently $C_1$, in a maximum of deg($C_2$)=$d_2$ points, so $d<2$ or $d<3$ respectively. Thus (\ref{points2}) does not hold in these cases. Further $C_1 \cup C_2$ is halfcanonical so $K_{C_2} + D = 2H$ for $H$ a hyperplane section. We also have for $C_2$ a plane curve of degree $d_2$ that $K_{C_2} = (-3+d_2)H$, and so $D = (5-d_2)H \leq H$ and the only solution is $d_2=4$. Since $C_2$ is defined by a nonsingular plane quartic it has degree 4 and genus 3, and by (\ref{points2}) the double locus of $C = C_1 \cup C_2$ contains 4 points. \\

Let $\Gamma$ be the curve $C_1 \cup l_1$ in $\mathbb{P}^4$, with $l_1$ defined by $x_0=x_1=x_2=0$. Again it follows that  $\Gamma$ is halfcanonical, hence Gorenstein. The curve $\Gamma$ is of degree 14 since $C_1$ must have degree 13 by (\ref{points1}). By Buchsbaum-Eisenbud~\cite{10.2307/2373926} if $\Gamma$ is Gorenstein codimension 3 then it is defined by Pfaffians. Since $\Gamma$ is degree 14, and we need seven cubics to define $C$, $\Gamma$ is defined by the $6 \times 6$ Pfaffians of a $7 \times 7$ skew-symmetric matrix. As $\Gamma$ contains the line $l_1$ it is necessary that every Pfaffian lies in the ideal $(x_0,x_1,x_2)$.

\begin{proposition}
    Let M be a $7 \times 7$ skew-symmetric matrix. Let $I$ be the ideal of $6 \times 6$ Pfaffians of $M$. Then two possible formats such that $I \subset (x_0,x_1,x_2)$ are as follows:
    \begin{flalign*}
    (\textup{I}) \hspace{3cm} M = \begin{pmatrix}
    & a_{12} & a_{13} & a_{14} & a_{15} & b_{16} & b_{17}\\
    & & a_{23} & a_{24} & a_{25} & b_{26} & b_{27} \\
    & & & a_{34} & a_{35} & b_{36} & b_{37} \\
    & & & & a_{45} & b_{46} & b_{47} \\
    & & & & & b_{56} & b_{57} \\
    & & & & & & b_{67} \\ 
    \end{pmatrix} \\ \\
\end{flalign*}
\begin{flalign*}
    (\textup{II}) \hspace{3cm} M = \begin{pmatrix}
    & b_{12} & b_{13} & b_{14} & b_{15} & a_{16} & a_{17}\\
    & & b_{23} & b_{24} & b_{25} & a_{26} & a_{27} \\
    & & & b_{34} & b_{35} & a_{36} & a_{37} \\
    & & & & b_{45} & a_{46} & a_{47} \\
    & & & & & a_{56} & a_{57} \\
    & & & & & & a_{67} \\ 
    \end{pmatrix} \\
\end{flalign*}
where the $a_{ij}$ represent linear elements in $(x_0,x_1,x_2)$ and the $b_{ij}$ represent linear elements in all the coordinates on $\mathbb{P}^4$.
\end{proposition}

\begin{proof}
    Consider the skew-symmetric matrix
    \begin{equation}
    M = \begin{pmatrix}
    & a_{12} & a_{13} & a_{14} & a_{15} & a_{16} & a_{17}\\
    & & a_{23} & a_{24} & a_{25} & a_{26} & a_{27} \\
    & & & a_{34} & a_{35} & a_{36} & a_{37} \\
    & & & & a_{45} & a_{46} & a_{47} \\
    & & & & & a_{56} & a_{57} \\
    & & & & & & a_{67} \\
    \end{pmatrix}.
\end{equation}
    We can explicitly state the Pfaffians of $M$, as in \cite{ishikawa2000minor}. Consider the $2n \times 2n$ submatrix defined by deleting the $i$th row and $i$th column of $M$. We define its Pfaffian using the symmetric group $G_i=S_{2n}$ on the set $\{1,\dots,\widehat{i},\dots,2n+1\}$. Let

\[
  \mathfrak{S_i} = \left\{  \sigma=(\sigma_1\dots \sigma_{2n}) \in G_i\ \quad \middle\vert \begin{array}{l}
    \quad \sigma_{2j-1}<\sigma_{2j} \qquad  1 \leq j \leq n \\
    \quad \sigma_{2j-1} < \sigma_{2j+1} \quad  1 \leq j \leq n-1
  \end{array}\right\}.
\]
    Recall that we can define the $2n \times 2n$ Pfaffian of the submatrix as 
    \begin{equation}
    \text{Pf}_i = \sum_{\sigma \in \mathfrak{S}_i}\text{sgn}(\sigma)a_{\sigma_1 \sigma_2}\cdots a_{\sigma_{2n-1}\sigma_{2n}}
    \end{equation}
    where $\text{sgn}(\sigma)=(-1)^{\ell (\sigma)}$, with $\ell (\sigma)$ the number of inversions.
    It follows that if all the $a_{\sigma_1 \sigma_2}\dots a_{\sigma_{2n-1} \sigma_{2n}}$ are contained in an ideal, then the Pfaffian must be contained in the ideal. We are working with $7 \times 7$ matrices so $n=3$ here and consequently each Pfaffian is a sum of elements of the form $a_{ij}a_{kl}a_{mn}$. It is clear that since any $a_{ij}a_{kl}a_{mn}$ contains three elements from different columns, (I) is a possible solution. Moreover, since any $a_{ij}a_{kl}a_{mn}$ must contain either a 6 or 7 in its indices, (II) also ensures $I \subset (x_0,x_1,x_2)$.
\end{proof}

More generally, we see that for any $(i,j) \in \{1,\dots,7\}$ with $i \neq j$ we can define two possible matrices such that $I \subset (x_0,x_1,x_2)$. Case (II) is the case where $a_{kl} \in (x_0,x_1,x_2)$ for $k \in \{i,j\}$ or $l \in \{i,j\}$ and other elements are general linear elements in the coordinates on $\mathbb{P}^4$. Case (I) is the case where $a_{kl} \in (x_0,x_1,x_2)$ for $k,l \notin \{i,j\}$ and other elements are general. In the language of Tom and Jerry, case (II) corresponds to $\text{Jer}_{67}$ and case (I) corresponds to $\text{Tom}_{67}$.\\

Defining $\Gamma$ using a matrix of the above form gives a curve with two irreducible nonsingular components: $C_1$ of degree 13 and genus 12 as expected, and $l_1$, the line defined by $(x_0,x_1,x_2)$ in $\mathbb{P}^4$. Moreover, the intersection of $C_1$ and $l_1$ is 4 points which define a quartic $q_4$ in $(x_3,x_4)$. This quartic can be mapped into $\mathbb{P}^2$ by adding arbitrary terms in $(x_5)$ to obtain the nonsingular quartic which defines $C_2$. The union of these two curves is a Gorenstein codimension 4 variety in $\mathbb{P}^5$ corresponding to Betti table 2.5.

\subsection{Type 2.8}\label{2.8}
We now consider type 2.8.
\begin{table}[h!]
    \[\begin{array}{c|l}
        & 0 \hspace{0.65cm} 1\hspace{0.65cm} 2 \hspace{0.68cm} 3 \hspace{0.65cm} 4 \\ \hline
       0 &1 \hspace{0.5cm} - \hspace{0.45cm} - \hspace{0.45cm} - \hspace{0.45cm} - \\
      1& - \hspace{0.55cm} 5 \hspace{0.65cm} 6 \hspace{0.65cm} 2 \hspace{0.6cm}-\\
      2&- \hspace{0.55cm} 2 \hspace{0.65cm} 4 \hspace{0.65cm} 2 \hspace{0.6cm} -\\
      3&- \hspace{0.55cm} 2 \hspace{0.65cm} 6 \hspace{0.65cm} 5 \hspace{0.6cm} - \\
      4& - \hspace{0.4cm} - \hspace{0.44cm} - \hspace{0.42cm} - \hspace{0.53cm} 1
    \end{array}
\]
\caption*{Type 2.8~\cite{schenck2020calabiyau}}
\label{tab:table3}
\end{table}
An example construction of a Gorenstein curve $C$ in $\mathbb{P}^5$ with such a minimal free resolution of its coordinate ring is as follows. Firstly, note that this variety is of degree 15, which follows from the Hilbert function \cite{schenck2020calabiyau}. Secondly the constraints of regularity 4 mean our curve will again be halfcanonical, of arithmetic genus 16. This ``big ears'' construction works in the following manner. Let $x_0,\dots,x_5$ be coordinates on $\mathbb{P}^5$. The quadrics
\begin{equation}
    Q_1=x_0x_4, \quad Q_2=x_1x_4, \quad Q_3=x_2x_5, \quad Q_4=x_3x_5, \quad Q_5=x_4x_5
\end{equation}
have six linear first syzygies and two linear second syzygies, thus satisfying the second row of the Betti table. It follows that $C=C_0 \cup C_1 \cup C_2$, with $C_0 \subset \mathbb{P}^3_{\left<x_0\dots x_3\right>}$, $C_1 \subset \mathbb{P}^2_{\left<x_0:x_1:x_5\right>}$, $C_2 \subset \mathbb{P}^2_{\left<x_2:x_3:x_4\right>}$. Each copy of $\mathbb{P}^2$ intersects the copy of $\mathbb{P}^3$ in a line, hence the term ``big ears'' to refer to the curves embedded into each $\mathbb{P}^2$. $C_0$ is a degree 7 genus 4 curve residual to a (3,3) complete intersection in $\mathbb{P}^3_{\left<x_0\dots x_3\right>}$ containing both lines, $l_1 \colon x_0=x_1=0$ and $l_2 \colon x_2=x_3=0$. Each cubic is thus in the ideal $J=(x_0x_2,x_1x_2,x_0x_3,x_1x_3)$. Let $I=(F_1,F_2)$ be the ideal defining the (3,3) complete intersection in $\mathbb{P}^3$, with 
\begin{equation}
\begin{split}
   & F_1=l_{13}x_1x_3+l_{23}x_2x_3+l_{14}x_1x_4+l_{24}x_2x_4, \\
   & F_2=m_{13}x_1x_3+m_{23}x_2x_3+m_{14}x_1x_4+m_{24}x_2x_4,
\end{split}
\end{equation}
where the $l_{ij},m_{ij}$ are linear forms in $k[x_0,\dots,x_3]$. The ideal defining the residual curve $C_0$ contains a further two quartics. These quartics may be calculated directly from the $2 \times 2$ minors of the matrix
\begin{equation}
M =
    \begin{pmatrix}
    l_{13} & l_{23} & l_{14} & l_{24} \\
    m_{13} & m_{23} & m_{14} & m_{24} \\
    \end{pmatrix}.
\end{equation}
One quartic, $H_1$ is double on the line $l_1$, and intersects $l_2$ transversally in 4 points, and vice versa for the second quartic $H_2$. Thus mapping $H_1$ into $\mathbb{P}^2_{\left<x_0:x_1:x_5\right>}$ by adding arbitrary terms in $(x_5)$ defines a nonsingular quartic curve $C_1$, and similarly mapping $H_2$ into $\mathbb{P}^2_{\left<x_2\dots x_4\right>}$ by adding arbitrary terms in $(x_4)$ defines a nonsingular quartic curve $C_2$. The union of these three curves is a codimension 4 Gorenstein curve in $\mathbb{P}^5$ whose coordinate ring has free resolution as in Betti table 2.8.

\bibliographystyle{plain}
\bibliography{references}

\end{document}